\documentclass[journal]{IEEEtran}
\usepackage{amsmath,amsfonts}
\usepackage{algorithmic}
\usepackage{algorithm}
\usepackage{array}
\usepackage[caption=false,font=normalsize,labelfont=sf,textfont=sf]{subfig}
\usepackage{textcomp}
\usepackage{stfloats}
\usepackage{url}
\usepackage{verbatim}
\usepackage{graphicx}
\usepackage[T1]{fontenc}
\usepackage{bm}
\usepackage{amsthm}
\usepackage{nomencl}
\usepackage{etoolbox}
\usepackage{ifthen}

\usepackage{booktabs}
\usepackage{multirow}
\usepackage{cite}
\newtheorem{proposition}{Proposition}

\makenomenclature

\renewcommand\nomgroup[1]{
\item[\itshape
\ifstrequal{#1}{A}{Sets and indices:}{
\ifstrequal{#1}{B}{Data:}{
\ifstrequal{#1}{C}{Variables:}{}}}
]}

\begin{document}
	
	\title{Variable aggregation-based formulations for pumped storage hydro model in the day-ahead unit commitment problem}
	
	\author{ Shaoze Li, Junhao Wu, and Zhibin Deng
		\thanks{ S. Li and Z. Deng is with the School of Economics and Management, University of Chinese Academy of Sciences; MOE Social Science Laboratory of Digital Economic Forecasts and Policy Simulation at UCAS, 
			Beijing 100190, China (e-mail: shaoze\_li@163.com, zhibindeng@ucas.edu.cn) J. Wu is with the School of Economics and Management, North China Electric Power University, Beijing, 102206, China
			(e-mail: junhao\_wu@163.com)..
		}
		
		% <-this % stops a space
		%\thanks{This paper was produced by the IEEE Publication Technology Group. They are in Piscataway, NJ.}% <-this % stops a space
		%\thanks{Manuscript received April 19, 2021; revised August 16, 2021.}
	}
	
	% The paper headers
	%\markboth{Journal of \LaTeX\ Class Files,~Vol.~14, No.~8, August~2021}%
	%{Shell \MakeLowercase{\textit{et al.}}: A Sample Article Using IEEEtran.cls for IEEE Journals}
	
	%\IEEEpubid{0000--0000/00\$00.00~\copyright~2021 IEEE}
	% Remember, if you use this you must call \IEEEpubidadjcol in the second
	% column for its text to clear the IEEEpubid mark.
	
\maketitle

\begin{abstract}
Pumped storage hydro (PSH) plants can improve the flexibility of power systems. A well-designed formulation for a PSH model is essential when incorporating the PSH units into a day-ahead unit commitment model. In the literature, the formulation of a PSH model is generally based on the individual PSH unit. This formulation is tight if there is only one PSH unit in the reservoir. However, when there are multiple units sharing the same reservoir in a PSH plant, the existing formulation may introduce some symmetric structures which degrade the efficiency of a mixed-integer programming solver significantly. In this paper, to cope with the symmetric structure in the PSH plants that have multiple units, we propose two new formulations. The first formulation considers the case in which there are multiple identical units sharing the same reservoir. The second formulation considers a general case in which the units sharing the same reservoir have the same generating and pumping efficiency. Using the new formulations, the symmetric structures in the problem can be effectively broken. Numerical results are presented to study the computational efficiency of the new formulations.
%
%Pumped storage hydro plants can improve the flexibility to power systems. Some recent papers have introduced pumped storage hydro plants into day-ahead market. A suitable formulation for a pumped storage hydro plant is and essential aspect when incorporating the pumped storage hydro plant in to a day-ahead unit commitment model. Generally, in a pumped storage hydro plant, there are multiple units that share the same reservoir. In the current literature, the formulation of a PSH unit is generally constructed individually. However, this will impact the computational performance of a mixed-integer programming solver significantly, due to the symmetric structure of the problem.
%
%In this paper, we propose two new formulations of PSH plants. The first formulation considers the case in which there are multiple identical units sharing the same reservoir. The second formulation considers a general case in which the units sharing the same reservoir are not identical. In comparing the existing formulations, the two proposed formulations does not model each single unit individually. In contrast, the formulations are derived by defining some aggregated variables, which only models the total number of units that are in different operation modes, and the total amount of generation power and total amount of pumping power. Numerical results are presented to study the impact of the new formulations on the computational efficiency of the model.
\end{abstract}
	

\nomenclature[A]{$t\in\mathcal{T}$}{set of periods}
\nomenclature[A]{$r\in\mathcal{R}$}{set of reservoirs}
\nomenclature[A]{$g\in\mathcal{G}_{r}$}{set of pumped storage hydro units in reservoir $r$}
\nomenclature[A]{$g\in\mathcal{G}$}{set of the rest of the generating units (other than the pumped storage hydro units) in a system}

\nomenclature[B]{$\alpha_{g}$}{generating efficiency of the PSH unit $g$}
\nomenclature[B]{$\beta_{g}$}{pumping efficiency of the PSH unit $g$}
\nomenclature[B]{$\underline{p}^{\textrm{gen}}_{g}$}{minimum  generation power of PSH unit $g$ [MW]}
\nomenclature[B]{$\bar{p}^{\textrm{gen}}_{g}$}{maximum generation power of PSH unit $g$ [MW]}
\nomenclature[B]{$\underline{p}^{\textrm{pump}}_{g}$}{minimum pumping power of PSH unit $g$ [MW]}
\nomenclature[B]{$\bar{p}^{\textrm{pump}}_{g}$}{maximum pumping power of PSH unit $g$ [MW]}
\nomenclature[B]{$\underline{S}_{r}$}{minimum  energy level of the reservoir $r$ [MWh]}
\nomenclature[B]{$\bar{S}_{r}$}{maximum energy level of the reservoir $r$ [MWh]}
\nomenclature[B]{$D_{t}$}{system net load at period $t$ [MW];}

\nomenclature[B]{$\Delta_g^{+}$}{the ramp-up limit for unit $g$ [MW];}
\nomenclature[B]{$\Delta_g^{-}$}{the ramp-down limit for unit $g$ [MW];}

\nomenclature[B]{$\Lambda_g^{+}$}{the start-up limit for unit $g$ [MW];}
\nomenclature[B]{$\Lambda_g^{-}$}{the shut-down limit for unit $g$ [MW];}

\nomenclature[B]{$\tau_g^{+}$}{the minimum number of consecutive time periods that unit $g\in\mathcal{G}$ has to be in the on state;}
\nomenclature[B]{$\tau_g^{-}$}{the minimum number of consecutive time periods that unit $g\in\mathcal{G}$ has to be in the off state}

\nomenclature[B]{$\Lambda_g^{-}$}{the shut-down limit for unit $g$ [MW];}

\nomenclature[C]{$u^{\textrm{gen}}_{g,t}$}{binary commitment variable that represents whether the unit $g$ is generating or not during time period $t$}
\nomenclature[C]{$u^{\textrm{pump}}_{g,t}$}{binary commitment variable that represents whether the unit $g$ is pumping or not during time period $t$}
\nomenclature[C]{$u^{\textrm{off}}_{g,t}$}{binary commitment variable that represents whether the unit $g$ is switched-off or not during time period $t$}
\nomenclature[C]{$p^{\textrm{gen}}_{g,t}$}{amount of generation at a PSH unit $g$ during time period $t$ [MW]}
\nomenclature[C]{$p^{\textrm{pump}}_{g,t}$}{amount of pumping at a PSH unit $g$ during time period $t$ [MW]}
\nomenclature[C]{$s_{r,t}$}{energy stored in the reservoir $r$ at the end of period $t$ [MWh]}
\nomenclature[C]{$w^{\textrm{pump}}_{r,t}$}{binary variable that represents whether the reservoir $r$ is in pumping mode or not during time period $t$}
\nomenclature[C]{$w^{\textrm{gen}}_{r,t}$}{binary variable that represents whether the reservoir $r$ is in generating mode or not during time period $t$}

\nomenclature[C]{$x_{g,t}$}{the commitment variable of unit $g\in\mathcal{G}$ during time period $t$}
\nomenclature[C]{$v_{g,t}$}{start-up variable which denoting if unit $g\in\mathcal{G}$ has been started-up at time period $t$}
\nomenclature[C]{$w_{g,t}$}{shut-down variable which denoting if unit $g\in\mathcal{G}$ has been shut-down at time period $t$}
\nomenclature[C]{$q_{g,t}$}{amount of generation at unit $g\in\mathcal{G}$ during time period $t$ [MW]}

\printnomenclature

\section{Introduction}\label{sec1}

Pumped storage hydro (PSH) plants can provide flexibility to enhance the reliability of modern power systems. It can absorb excess electricity when the net load is low and generate electricity when it is needed. With increasing penetration of renewable energy, the PSH plant plays an important role to address uncertainties in modern power system \cite{Roberts}.

The PSH plant has a state of charge (SOC). It can pump to increase its SOC, and generate power to decrease its SOC. From a computational perspective, the main challenge of modelling a PSH plant comes from the combination of the mutually exclusive charging and discharging modes, together with SOC limits. To address this issue, a mixed-integer formulation is generally needed. The key question is how to derive a formulation to model a PSH plant, so that the formulation can bring benefit to the computational efficiency of a solver.

In the literature, various formulations for a PSH plant have been proposed and have been integrated into many models, such as the profit maximization models \cite{Castronuovo2004a,Duque}, energy storage sizing models \cite{Korpaas,Castronuovo2004b,Brown}, bidding and scheduling models \cite{Ni,Gonzalez,Li2005}, and so on.
However, the computational efficiency of incorporating a PSH plant in an unit commitment problem is not the major concern of the aforementioned papers. Some regional transmission organizations and independent system operators have optimized pumped storage hydro in the day ahead market \cite{Giacomoni}.
In \cite{Jiang}, a robust unit commitment problem with wind power and a PSH unit is studied.  In \cite{Khodayar} and \cite{Li2016}, the scheduling of PSH units are modeled in a day-ahead market. In \cite{Huang2022}, a mixed-integer formulation of PSH plant is proposed and this formulation is integrated into a unit-commitment problem for the day-ahead market. In order to improve the computational efficiency of the mixed-integer formulation, \cite{Baldick} derived some new valid inequalities to tighten the formulation proposed in \cite{Huang2022}.

The formulations in \cite{Huang2022} and \cite{Baldick} are all based on individual PSH unit. In detail, for each PSH unit, a set of binary variables and a set of continuous variables are introduced to represent the operation modes and the amount of load/generation in each time period. This formulation is compact for a PSH plant that has one unit. However, in real-life applications, it is common that there are multiple PSH units sharing one reservoir. If some of the PSH units in the same reservoir are identical (in other word, the parameters of these PSH units are the same),  then the individual PSH unit based mixed-integer formulations will have certain types of symmetric structures, which may lead to significant degradation to the computational efficiency of a solver such as Gurobi. Our numerical results will show that even when there are as few as three identical PSH units sharing one reservoir, the performance of Gurobi will deteriorate significantly if the individual PSH unit based formulations in \cite{Huang2022} and \cite{Baldick} are applied.

The main concern of this paper is to design new formulations of a PSH plant that can break the symmetric structures in the problem.
In the literature, various techniques have been proposed to break the symmetric structure. For example, the most straightforward method is to define the representative solution among different symmetric solutions, and add some cuts in the mixed-integer programming problems to cut-off other non-representative symmetric solutions, while keeping the representative solutions feasible \cite{Friedman,Liberti2012,Liberti2014}. Besides, an advanced branching rule called as orbital branching rule, has been proposed, which only enumerates the representative solutions in the branching process \cite{Ostrowski2011,Ostrowski2015}. For the unit commitment problems with identical units, an aggregated variables based method is proposed to break the symmetric structure of the solution space \cite{Knueven2017}.

In this paper, we consider two types of symmetric structures in PSH units. One is called as full symmetric structure, in which all the PSH units in one reservoir are identical. The other is called as partial symmetric structure, in which only the generating and pumping efficiency of the PSH units sharing one reservoir are identical. To handle the two types of symmetric structures, we propose two formulations, one is based on variable aggregation, and the other is based on problem presolving. Our numerical results will show the effectiveness of the proposed formulations on handling the symmetric structures.

The remaining parts of this paper are arranged as follows: Section \ref{sec2} reviews a formulation based on the individual PSH unit. Section \ref{sec3} derives two new formulations. One is derived by using variable aggregation, which is designed for breaking the symmetric structures in the cases where the PSH units sharing the same reservoir are identical. The other is derived via a presolving step, which is designed for the cases that the PSH units in the same reservoir are not identical, but have the same generating and pumping efficiency. Some numerical results on studying the value of different PSH formulations to the computational efficiency of a mixed-integer programming solver are presented in Section \ref{sec4}.

\section{The Standard Formulations}\label{sec2}
In this section, we introduce the PSH formulation introduced in \cite{Huang2022} and \cite{Baldick}. This formulation will be referred to as the standard formulation in this paper.

In the standard formulation, for each unit $g\in \mathcal{G}_r$, we use $p^{\textrm{gen}}_{g,t}$ and $p^{\textrm{pump}}_{g,t}$ to represent the amount of generation and the amount of pumping load of unit $g$ during period $t$, respectively, and use binary variables $u^{\textrm{gen}}_{g,t}$, $u^{\textrm{pump}}_{g,t}$ and $u^{\textrm{off}}_{g,t}$ to represent the three operation modes of a unit, namely generating, pumping and off-line, respectively.

For each unit $g$ in each period $t$, the operation modes are mutually exclusive. This is represented by
\begin{equation}\label{eq1}
u^{\textrm{gen}}_{g,t}+u^{\textrm{pump}}_{g,t}+u^{\textrm{off}}_{g,t}=1,~\forall g\in \mathcal{G}_r,t\in\mathcal{T}.
\end{equation}
The amount of pumping load and the amount of generation in each period are constrained by their lower and upper bounds. These are represented by the following constraints:
\begin{equation}\label{eq2}
\begin{aligned}
&\underline{p}^{\textrm{gen}}_g u^{\textrm{gen}}_{g,t} \leq p^{\textrm{gen}}_{g,t}\leq \bar{p}^{\textrm{gen}}_g u^{\textrm{gen}}_{g,t},~& g\in \mathcal{G}_r,t\in\mathcal{T},\\
&\underline{p}^{\textrm{pump}}_g u^{\textrm{pump}}_{g,t}\leq p^{\textrm{pump}}_{g,t}\leq \bar{p}^{\textrm{pump}}_g u^{\textrm{pump}}_{g,t},~& g\in \mathcal{G}_r,t\in\mathcal{T}.
\end{aligned}
\end{equation}

For a typical PSH plant, it is general that there are multiple PSH units installed in one reservoir.
When there are more than one PSH units sharing the same reservoir, it is not economical and physically feasible to have some units pumping while other units generating simultaneously. To impose these mutual exclusivity constraints in reservoir $r$ at the plant level, two binary variables $w^{\textrm{pump}}_{r,t}$ and $w^{\textrm{gen}}_{r,t}$ are introduced for each period $t\in \mathcal{T}$, which represent the operation status of the plant as pumping and generating at period $t$, respectively. Then, the following constraints are introduced:
\begin{equation}\label{eq3}
w^{\textrm{pump}}_{r,t}+w^{\textrm{gen}}_{r,t}\leq 1,~t\in\mathcal{T}.
\end{equation}
The operation mode of each unit in reservoir $r$ is constrained by the operation status of the plant as follows:
\begin{equation}\label{eq4}
\begin{aligned}
&u_{g,t}^{\textrm{pump}}\leq w_{r,t}^{\textrm{pump}},~&g\in \mathcal{G}_r,t\in\mathcal{T},\\
&u_{g,t}^{\textrm{gen}}\leq w_{r,t}^{\textrm{gen}},~&g\in \mathcal{G}_r,t\in\mathcal{T}.
\end{aligned}
\end{equation}

For a reservoir $r\in \mathcal{R}$, the SOC variable $s_{r,t}$ is introduced for each time period $t\in \mathcal{T}$. The SOC constraints are introduced as follows:
\begin{equation}\label{eq5}
s_{r,t}=s_{r,t-1}+\sum_{g\in \mathcal{G}_r} \alpha_g p^{\textrm{pump}}_{g,t-1}-\sum_{g\in \mathcal{G}_r} \beta_g p^{\textrm{gen}}_{g,t-1},~t\in \mathcal{T},
\end{equation}
where $s_{r,0}$ is fixed to a constant $S_{r,0}$, which represents a known initial condition. Besides, in some formulations, there may be a specified level $S_{r,T}$ for the final state $s_{r,T}$. The constraints on the initial and final conditions are given as follows:
\begin{equation}\label{eq6}
\begin{aligned}
&s_{r,0}=S_{r,0},\\
&s_{r,T}=S_{r,T}.
\end{aligned}
\end{equation}
There is a maximum allowable SOC level $\bar{S}_r$ and a minimum allowable SOC level $\underline{S}_r$. Hence the value of $s_{r,t}$ is bounded as follows:
\begin{equation}\label{eq7}
\underline{S}_r\leq s_{r,t}\leq \bar{S}_r,~t\in \mathcal{T}.
\end{equation}
Moreover, as discussed in  \cite{Baldick}, based on the SOC constraints (5) and the mutual exclusivity constraints \eqref{eq3} and \eqref{eq4}, the constraints in \eqref{eq7} can be tightened as follows:
\begin{equation}\label{eq8}
\begin{aligned}
&s_{r,t-1}+\sum_{g\in \mathcal{G}_r} \alpha_g p^{\textrm{pump}}_{g,t-1}\leq \bar{S}_r, &t\in \mathcal{T},\\
&s_{r,t-1}-\sum_{g\in \mathcal{G}_r} \beta_g p^{\textrm{gen}}_{g,t-1}\geq \underline{S}_r, &t\in \mathcal{T}.
\end{aligned}
\end{equation}
Given the constraints in \eqref{eq8}, the constraints in \eqref{eq7} are redundant and can be dropped. The effects of using \eqref{eq8} to replace \eqref{eq7} have been analyzed in details in Section 3.3 of  \cite{Baldick}.

Now we give the complete standard PSH formulation. For simplicity on notations, we use vectors $\textbf{u}_r$, $\textbf{p}_r$, $\textbf{s}_r$ and $\textbf{w}_r$ to denote the collections of variables that represent the operation modes, the amount of generation and the amount of pumping load, the state-of-charge, and the operation status of the plant, respectively:
\begin{equation}\label{eq9}
\begin{aligned}
&\textbf{u}_r=\left\{u^{\textrm{gen}}_{g,t},u^{\textrm{pump}}_{g,t},u^{\textrm{off}}_{g,t}\right\}_{g\in \mathcal{G}_r,t\in \mathcal{T}},\\
&\textbf{p}_r=\left\{p^{\textrm{gen}}_{g,t},p^{\textrm{pump}}_{g,t}\right\}_{g\in \mathcal{G}_r,t\in \mathcal{T}},\\
&\textbf{s}_r=\left\{s^{\textrm{gen}}_{r,t}\right\}_{t\in \mathcal{T}},\\
&\textbf{w}_r=\left\{w^{\textrm{gen}}_{r,t}\right\}_{t\in \mathcal{T}}.
\end{aligned}
\end{equation}
Then, the standard PSH formulation for reservoir $r$ is given as follows:
\begin{equation}\label{eq10}
\mathcal{S}_{r}=\left\{ (\textbf{u}_r,\textbf{p}_r,\textbf{s}_r,\textbf{w}_r) \,\left|\,
\begin{array}{@{}lll}
& \eqref{eq1}-\eqref{eq6},~\textrm{and}~\eqref{eq8}\\[3pt]
& \textbf{u}_r\in \{0,1\}^{3T\times|\mathcal{G}_r|} \\[3pt]
&\textbf{w}_r\in \{0,1\}^{2T}
\end{array}\right.\right\}.
\end{equation}

The formulation of a PSH plant can be incorporated into a unit-commitment problem.
For each generator $g\in \mathcal{G}$ other than the PSH units, we use a binary vector $\textbf{u}_{g}$ to denote the collection of variables that represent the operation status of the unit, and a continuous vector $\textbf{q}_{g}$ to denote the collection of variables that represent the amount of generation over each time period. The constraints on the vectors $\textbf{u}_{g}$ and $\textbf{q}_{g}$, such as ramping constraints, start-up and shut-down time constraints, minimum up/down time constraints and some other constraints are denoted as $(\textbf{u}_{g},\textbf{q}_{g})\in \mathcal{X}_g$. In the literature, various formulations for the unit-commitment problem have been proposed \cite{Ostrowski2012,Espana2013,Gentile,Espana2015,Knueven,Bacci,Rajan,Reddy,Shiina}.  We adopt the 3-bin formulation in \cite{Rajan} to represent the set $\mathcal{X}_g$. Following the formulation in \cite{Huang2022}, the unit commitment problem with PSH units being incorporated can be formulated as follows:
\begin{equation}\label{eq11}
\begin{aligned}
\min~&\sum_{g\in\mathcal{G}} C(\textbf{u}_{g},\textbf{q}_{g})\\
\textrm{s.t.}~&\sum_{g\in\mathcal{G}} q_{g,t}+ \sum_{r\in\mathcal{R}}\sum_{g\in\mathcal{G}_r} (p^{\textrm{gen}}_{g,t}- p^{\textrm{pump}}_{g,t})=D_t,~t\in\mathcal{T},\\
&(\textbf{u}_r,\textbf{p}_r,\textbf{s}_r,\textbf{w}_r)\in \mathcal{S}_r,~r\in\mathcal{R},\\
&(\textbf{u}_{g},\textbf{q}_{g})\in \mathcal{X}_g,~g\in\mathcal{G}.
\end{aligned}
\end{equation}
The objective is to minimize the total production cost of the rest generators in the system. The detailed formulation of $\mathcal{X}_g$ and the formulation of the cost function $C(\textbf{u}_{g},\textbf{q}_{g})$ of a unit $g\in\mathcal{G}$ will be presented in the Appendix of this paper. Since the operation cost of the PSH units is much lower than the cost of the rest units, their cost is ignored in the objective function.

\section{Two New Formulations}\label{sec3}

When more than one PSH units share the same reservoir, it is common that some parameters of these PSH units are identical. In this case, due to the symmetric structures in the mixed-integer programming problem, the efficiency of a solver will degrade significantly if the standard PSH formulation is used in the unit commitment model. To cope with this issue, we consider two types of symmetric structures in the standard formulation, and propose two new formulations to break the symmetric structures.

\subsection{A Formulation for Identical Units in One Reservoir}

The first type of symmetric structure we consider will be called as full symmetric structure, in which all parameters of the PSH units in the same reservoir are identical. In detail, the set of parameters of a PSH unit is represented by
$$\{\underline{p}^{\textrm{gen}}_g,~\bar{p}^{\textrm{gen}}_g,~\underline{p}^{\textrm{pump}}_g,~\bar{p}^{\textrm{pump}}_g,~\alpha_g,\beta_g\}.$$
%We assume that all the PSH units in reservoir $r$ have the same parameters. The identical parameters are denoted as
Since we have assumed that all the PSH units in reservoir $r$ have the same parameters, the identical parameters can be denoted as
$$\{\underline{p}^{\textrm{gen}}_r,~\bar{p}^{\textrm{gen}}_r,~\underline{p}^{\textrm{pump}}_r,~\bar{p}^{\textrm{pump}}_r,~\alpha_r,\beta_r\}.$$

To derive the aggregated formulation, we define the following aggregate variables for each period $t\in \mathcal{T}$:
\begin{equation}\label{eq12}
\begin{aligned}
&P^{\textrm{gen}}_{r,t}=\sum_{g\in \mathcal{G}_r} p^{\textrm{gen}}_{g,t},~
&P^{\textrm{pump}}_{r,t}=\sum_{g\in \mathcal{G}_r} p^{\textrm{pump}}_{g,t},\\
&U^{\textrm{gen}}_{r,t}=\sum_{g\in \mathcal{G}_r} u^{\textrm{gen}}_{g,t},~
&U^{\textrm{pump}}_{r,t}=\sum_{g\in \mathcal{G}_r} u^{\textrm{pump}}_{g,t},\\
&U^{\textrm{off}}_{r,t}=\sum_{g\in \mathcal{G}_r} u^{\textrm{off}}_{g,t}.
\end{aligned}
\end{equation}
It is clear that $U^{\textrm{gen}}_{r,t}$, $U^{\textrm{pump}}_{r,t}$ and $U^{\textrm{off}}_{r,t}$ denote the total numbers of PSH units in reservoir $r$ that are in the modes of generating, pumping and off-line during period $t$, respectively, and $P^{\textrm{gen}}_{r,t}$ and $P^{\textrm{pump}}_{r,t}$ represent the total amount of generation
and the total amount of pumping load of all the PSH units in reservoir $r$ during period $t$, respectively. While the idea of aggregation is not new \cite{Knueven2017,Sen,Gonzalez,Shortt,Palmintier}, the purpose of our work is to use the aggregate variables to replace the variables of individual PSH units to break the symmetric structures.

To replace the variables of individual PSH units, we change the constraints in \eqref{eq1}, \eqref{eq2}, \eqref{eq4}, \eqref{eq5} and \eqref{eq8} into an aggregated version. From \eqref{eq1}, we can derive the following constraint:
\begin{equation}\label{eq13}
U^{\textrm{gen}}_{r,t}+U^{\textrm{pump}}_{r,t}+U^{\textrm{off}}_{r,t}=n_r,~t\in \mathcal{T},
\end{equation}
where $n_r$ denotes the number of PSH units in reservoir $r$. From \eqref{eq2}, we can derive the following constraints:
\begin{equation}\label{eq14}
\begin{aligned}
&\underline{p}^{\textrm{gen}}_r U^{\textrm{gen}}_{r,t} \leq P^{\textrm{gen}}_{r,t}\leq \bar{p}^{\textrm{gen}}_r U^{\textrm{gen}}_{r,t},~&t\in\mathcal{T},\\
&\underline{p}^{\textrm{pump}}_r U^{\textrm{pump}}_{r,t}\leq P^{\textrm{pump}}_{r,t}\leq \bar{p}^{\textrm{pump}}_r U^{\textrm{pump}}_{r,t},~&t\in\mathcal{T}.
\end{aligned}
\end{equation}
Following the constraints in \eqref{eq4}, the aggregated variables are constrained as follows:
\begin{equation}\label{eq15}
\begin{aligned}
&U_{r,t}^{\textrm{pump}}\leq n_r w_{r,t}^{\textrm{pump}},&t\in \mathcal{T},\\
&U_{r,t}^{\textrm{gen}}\leq n_r w_{r,t}^{\textrm{gen}},&t\in \mathcal{T}.
\end{aligned}
\end{equation}
Finally, the constraints in \eqref{eq5} can be reformulated as follows:
\begin{equation}\label{eq16}
\begin{aligned}
&s_{r,t}=s_{r,t-1}+ \alpha_r P^{\textrm{pump}}_{r,t-1}-\beta_r P^{\textrm{gen}}_{r,t-1},~ t\in \mathcal{T},
\end{aligned}
\end{equation}
and the constraints in \eqref{eq8} can be reformulated as follows:
\begin{equation}\label{eq17}
\begin{aligned}
&s_{r,t-1}+\alpha_r P^{\textrm{pump}}_{r,t-1}\leq \bar{S}_r,&t\in \mathcal{T},\\
&s_{r,t-1}-\beta_r P^{\textrm{gen}}_{r,t-1}\geq \underline{S}_r,&t\in \mathcal{T}.
\end{aligned}
\end{equation}

To give the complete aggregated PSH formulation, we use $(\textbf{U}_r,\textbf{P}_r,\textbf{s}_r,\textbf{w}_r)$ to represent the set of variables in the aggregated formulation, where $\textbf{U}_r$ and $\textbf{P}_r$ are the collections of variables defined as follows:
\begin{equation}\label{eq18}
\begin{aligned}
&\textbf{U}_r=\left\{U^{\textrm{gen}}_{r,t},U^{\textrm{pump}}_{r,t},U^{\textrm{off}}_{r,t}\right\}_{t\in \mathcal{T}},\\
&\textbf{P}_r=\left\{P^{\textrm{gen}}_{r,t},P^{\textrm{pump}}_{r,t}\right\}_{t\in \mathcal{T}},
\end{aligned}
\end{equation}
and $\textbf{s}_r$ and $\textbf{w}_r$ are defined the same as those in \eqref{eq9}.
Then, the aggregated PSH formulation is given as follows:
\begin{equation}\label{eq19}
\mathcal{S}^\prime_{r}=\left\{ (\textbf{U}_r,\textbf{P}_r,\textbf{s}_r,\textbf{w}_r) \,\left|\,
\begin{array}{@{}lll}
& \eqref{eq3},~\eqref{eq6},~\eqref{eq13}-\eqref{eq17}\\[3pt]
& \textbf{U}_r\in \mathbb{Z}^{3T} \\[3pt]
&\textbf{w}_r\in \{0,1\}^{2T}
\end{array}\right.\right\}.
\end{equation}

In the aggregated PSH formulation, we keep the aggregated variables $\textbf{U}_r$ and $\textbf{P}_r$ to represent the overall states of all the units in reservoir $r$, but ignore the variables $\textbf{u}_r$ and $\textbf{p}_r$ that represent the states of each individual unit. Actually, given a solution $(\textbf{U}_r,\textbf{P}_r,\textbf{s}_r,\textbf{w}_r)\in \mathcal{S}^{\prime}_r$, we can recover a solution $(\textbf{u}_r,\textbf{p}_r,\textbf{s}_r,\textbf{w}_r)\in \mathcal{S}_r$ according to the following steps:
For each $t\in\mathcal{T}$, if $U^{\textrm{pump}}_{r,t}>0$ and $U^{\textrm{gen}}_{r,t}=0$, then we arbitrarily select a subset $\mathcal{S}$ of $\mathcal{G}_r$ that contains $U^{\textrm{pump}}_{r,t}$ elements, and set
\begin{equation}\label{eq20}
\begin{aligned}
&\left(u^{\textrm{pump}}_{g,t},u^{\textrm{gen}}_{g,t},u^{\textrm{off}}_{g,t}\right)=\left\{\begin{array}{@{}ll}
(1,0,0),~&\textrm{if}~g\in\mathcal{S},\\%g=\{1,\ldots,U^{\textrm{pump}}_{r,t}\},\\[3pt]
(0,0,1),~&\textrm{if}~g\in \mathcal{G}_r\setminus \mathcal{S},%g=\{U^{\textrm{pump}}_{r,t}+1,\ldots,|\mathcal{G}_r|\}.
\end{array}\right.\\
%\end{aligned}
%\end{equation}
%\noindent \textrm{and}\\
%\begin{equation}\label{eq20b}
%\begin{aligned}
&\left(p^{\textrm{pump}}_{g,t},p^{\textrm{gen}}_{g,t}\right)=\left\{\begin{array}{@{}ll}
(P^{\textrm{pump}}_{r,t}/U^{\textrm{pump}}_{r,t},0),~&\textrm{if}~g\in\mathcal{S},\\[3pt]
(0,0),~&\textrm{if}~g\in \mathcal{G}_r\setminus \mathcal{S}.
\end{array}\right.
\end{aligned}
\end{equation}
If $U^{\textrm{pump}}_{r,t}=0$ and $U^{\textrm{gen}}_{r,t}>0$, then we arbitrarily select a subset $\mathcal{S}$ of $\mathcal{G}_r$ that contains $U^{\textrm{gen}}_{r,t}$ elements, and set
\begin{equation}\label{eq21}
\begin{aligned}
&\left(u^{\textrm{pump}}_{g,t},u^{\textrm{gen}}_{g,t},u^{\textrm{off}}_{g,t}\right)=\left\{\begin{array}{@{}ll}
(0,1,0),~&\textrm{if}~g\in\mathcal{S},\\%g=\{1,\ldots,U^{\textrm{pump}}_{r,t}\},\\[3pt]
(0,0,1),~&\textrm{if}~g\in \mathcal{G}_r\setminus \mathcal{S},%g=\{U^{\textrm{pump}}_{r,t}+1,\ldots,|\mathcal{G}_r|\}.
\end{array}\right.\\
%\end{aligned}
%\end{equation}
%\noindent \textrm{and}\\
%\begin{equation}\label{eq21b}
%\begin{aligned}
&\left(p^{\textrm{pump}}_{g,t},p^{\textrm{gen}}_{g,t}\right)=\left\{\begin{array}{@{}ll}
(0,P^{\textrm{gen}}_{r,t}/U^{\textrm{gen}}_{r,t}),~&\textrm{if}~g\in\mathcal{S},\\[3pt]
(0,0),~&\textrm{if}~g\in \mathcal{G}_r\setminus \mathcal{S}.
\end{array}\right.
\end{aligned}
\end{equation}
If both $U^{\textrm{pump}}_{r,t}=0$ and $U^{\textrm{gen}}_{r,t}=0$, then we set $u^{\textrm{off}}_{g,t}=1$ for all $g\in \mathcal{G}_r$ and set all the other variables to zero.

Thus, for any $(\textbf{U}_r,\textbf{P}_r,\textbf{s}_r,\textbf{w}_r)\in\mathcal{S}^{\prime}_r$, we can recover a solution $(\textbf{u}_r,\textbf{p}_r,\textbf{s}_r,\textbf{w}_r)\in\mathcal{S}_r$. Obviously, the solution that can be recovered is not unique in general, since the subset $\mathcal{S}$ of $\mathcal{G}_r$ can be selected arbitrarily. In this way, the symmetric structure in
$\mathcal{S}_r$, in which there are multiple solutions that correspond to the same aggregated solution in $\mathcal{S}^{\prime}_r$, can be broken via variable aggregation. Thus, using the aggregated formulation can be effective for avoiding unnecessary enumerations in the branch-and-bound algorithm when solving the mixed-integer programming problem using a solver such as Gurobi.

Using the aggregated formulation instead of the standard formulation, the unit commitment problem in \eqref{eq11} can be reformulated as follows:

\begin{equation}\label{eq22}
\begin{aligned}
\min~&\sum_{g\in\mathcal{G}} C(\textbf{u}_{g},\textbf{q}_{g})\\
\textrm{s.t.}~&\sum_{g\in\mathcal{G}} q_{g,t}+ \sum_{r\in\mathcal{R}} (P^{\textrm{gen}}_{r,t}- P^{\textrm{pump}}_{r,t})=D_t,~t\in\mathcal{T},\\
&(\textbf{U}_r,\textbf{P}_r,\textbf{s}_r,\textbf{w}_r)\in\mathcal{S}^{\prime}_r,~r\in\mathcal{R},\\
&(\textbf{u}_{g},\textbf{q}_{g})\in \mathcal{X}_g,~g\in\mathcal{G}.
\end{aligned}
\end{equation}

\subsection{A Formulation via Presolving}

The second type of symmetric structure we consider is called as partial symmetric structure, in which only the generating efficiency $\alpha_g$ and pumping efficiency $\beta_g$ of the PSH units in the same reservoir are identical.
%not all the parameters of the PSH units in the same reservoir are identical, but only the generating efficiency $\alpha_g$ and pumping efficiency $\beta_g$ are identical.
Our numerical studies in Section \ref{sec4} will show that even when the unit commitment problem has a partial symmetric structure, the efficiency of a mixed-integer programming solver will degrade significantly if the standard PSH formulation is applied. Note that the aggregated PSH formulation is not applicable to this case, since the PSH units are not identical. To cope with this type of symmetric structure, we propose another formulation called as presolved PSH formulation in this subsection.

To derive the new formulation, we first define the following set:
\begin{equation}\label{eq23}
\mathcal{P}^{\textrm{pump}}_{r,t}=\left\{P^{\textrm{pump}}_{r,t}\,\left|\,
\begin{array}{@{}lll}
&P^{\textrm{pump}}_{r,t}=\sum_{g\in \mathcal{G}_r} p^{\textrm{pump}}_{g,t}\\[3pt]
& \underline{p}^{\textrm{pump}}_g u^{\textrm{pump}}_{g,t}\leq p^{\textrm{pump}}_{g,t}\leq \bar{p}^{\textrm{pump}}_g u^{\textrm{pump}}_{g,t}\\[3pt]
& \sum_{g\in \mathcal{G}_r} u^{\textrm{pump}}_{g,t}\geq 1 \\[3pt]
& u^{\textrm{pump}}_{g,t}\in\{0,1\},g\in \mathcal{G}_r
\end{array}\right.\right\}.
\end{equation}
For simplicity, we use
$$\textbf{u}^{\textrm{pump}}_{r,t}:=\left\{u^{\textrm{pump}}_{g,t}\right\}_{g\in \mathcal{G}_r}$$
to represent the collection of variables that represent the pumping states of all the PSH units in reservoir $r$ in period $t\in \mathcal{T}$.
Meanwhile, we define the set
\begin{equation}\label{eq24}
\mathcal{U}^{\textrm{pump}}_{r,t}=\left\{ \textbf{u}^{\textrm{pump}}_{r,t} \,\left|\,
\begin{array}{@{}lll}
& \sum_{g\in \mathcal{G}_r} u^{\textrm{pump}}_{g,t}\geq 1 \\[3pt]
& u^{\textrm{pump}}_{g,t}\in\{0,1\},g\in \mathcal{G}_r
\end{array}\right.\right\}
\end{equation}
to represent all the possible states in which there are at least one PSH unit being pumping in period $t$. It is obvious that there are $2^{n_r}-1$ different states in $\mathcal{U}^{\textrm{pump}}_{r,t}$.
Using these notations, the set $\mathcal{P}^{\textrm{pump}}_{r,t}$ can be defined by the union of the intervals each of which corresponds to one of the $2^{n_r}-1$ different states:
\begin{equation}\label{eq25}
\mathcal{P}^{\textrm{pump}}_{r,t}=\bigcup_{ \textbf{u}^{\textrm{pump}}_{r,t}\in \mathcal{U}^{\textrm{pump}}_{r,t}} \left[\sum_{g\in \mathcal{G}_r} \underline{p}^{\textrm{pump}}_g u^{\textrm{pump}}_{g,t},\sum_{g\in \mathcal{G}_r} \bar{p}^{\textrm{pump}}_g u^{\textrm{pump}}_{g,t}\right].
\end{equation}

From \eqref{eq23}--\eqref{eq25}, it is obvious to see that the definition of the set $\mathcal{P}^{\textrm{pump}}_{r,t}$ is independent from the index $t\in \mathcal{T}$ (in other words, for different time period $t\in \mathcal{T}$, the set $\mathcal{P}^{\textrm{pump}}_{r,t}$ is actually identical). Thus, we omit the subscript $t$ and denote the set $\mathcal{P}^{\textrm{pump}}_{r,t}$ as $\mathcal{P}^{\textrm{pump}}_{r}$ for simplicity. In general, the $2^{n_r}-1$  intervals in the right side of \eqref{eq25} may intersect with each other. If any two intervals have a common point, then the union of the two intervals can be merged into one interval. Hence, by merging the intervals that have common points iteratively, the union of the $2^{n_r}-1$ intervals can be reformulated as the union of $k_r$ disjoint intervals:
\begin{equation}\label{eq26}
\mathcal{P}^{\textrm{pump}}_{r}=\bigcup_{i=1}^{k_r} \left[l^{\textrm{pump}}_{r,i},u^{\textrm{pump}}_{r,i}\right].
\end{equation}
where $k_r$ is at most $2^{n_r}-1$.
To formulate $P^{\textrm{pump}}_{r,t}\in \mathcal{P}^{\textrm{pump}}_r$ as mixed-integer linear constraints, we introduce $k_r$ binary variables $\{z^{\textrm{pump}}_{r,t,i}\}_{i=1,\ldots,k_r}$, each of which corresponds to one interval in \eqref{eq26} and represents whether $P^{\textrm{pump}}_{r,t}$ belongs to that interval. Using these binary variables, the constraint $P^{\textrm{pump}}_{r,t}\in \mathcal{P}^{\textrm{pump}}_r$ can be represented as follows:
\begin{equation}\label{eq27}
\begin{aligned}
&P^{\textrm{pump}}_{r,t}=\sum_{i=1}^{k_r} X^{\textrm{pump}}_{r,t,i},\\
&l^{\textrm{pump}}_{r,i} z^{\textrm{pump}}_{r,t,i} \leq X^{\textrm{pump}}_{r,t,i}\leq u^{\textrm{pump}}_{r,i} z^{\textrm{pump}}_{r,t,i},\\
&\sum_{i=1}^{k_r} z^{\textrm{pump}}_{r,t,i}=w^{\textrm{pump}}_{r,t}.
\end{aligned}
\end{equation}

Similar to \eqref{eq23}, we define the following set to represent the possible total amount of generation of all the PSH units in reservoir $r$ in one time period:
\begin{equation}\label{eq28}
\mathcal{P}^{\textrm{gen}}_{r,t}=\left\{P^{\textrm{gen}}_{r,t}\,\left|\,
\begin{array}{@{}lll}
&P^{\textrm{gen}}_{r,t}=\sum_{g\in \mathcal{G}_r} u^{\textrm{gen}}_{g,t}\\[3pt]
& \underline{p}^{\textrm{gen}}_g u^{\textrm{gen}}_{g,t}\leq p^{\textrm{gen}}_{g,t}\leq \bar{p}^{\textrm{gen}}_g u^{\textrm{gen}}_{g,t}\\[3pt]
& \sum_{g\in \mathcal{G}_r} u^{\textrm{gen}}_{g,t}\geq 1 \\[3pt]
& u^{\textrm{gen}}_{g,t}\in\{0,1\},g\in \mathcal{G}_r
\end{array}\right.\right\}.
\end{equation}
We also ignore the subscript $t$ and simply denote $\mathcal{P}^{\textrm{gen}}_{r,t}$ as $\mathcal{P}^{\textrm{gen}}_{r}$. Using a similar derivation as in \eqref{eq23}--\eqref{eq26}, we can reformulate $\mathcal{P}^{\textrm{gen}}_{r}$ as a union of $k^\prime_r$ disjoint intervals:
\begin{equation}\label{eq29}
\mathcal{P}^{\textrm{gen}}_{r}=\bigcup_{i=1}^{k^\prime_r} \left[l^{\textrm{gen}}_{r,i},u^{\textrm{gen}}_{r,i}\right].
\end{equation}
Then, by introducing $k^\prime_r$ binary variables $\{z^{\textrm{gen}}_{r,t,i}\}_{i=1,\ldots,k^\prime_r}$, the constraint $P^{\textrm{gen}}_{r,t}\in \mathcal{P}^{\textrm{gen}}_{r}$ can be represented as follows:
\begin{equation}\label{eq30}
\begin{aligned}
&P^{\textrm{gen}}_{r,t}=\sum_{i=1}^{k_r} X^{\textrm{gen}}_{r,t,i},\\
&l^{\textrm{gen}}_{r,i} z^{\textrm{gen}}_{r,t,i} \leq X^{\textrm{gen}}_{r,t,i}\leq u^{\textrm{gen}}_{r,i} z^{\textrm{gen}}_{r,t,i},\\
&\sum_{i=1}^{k^\prime_r} z^{\textrm{gen}}_{r,t,i}=w^{\textrm{gen}}_{r,t}.
\end{aligned}
\end{equation}

To describe the complete presolved formulation, we use $(\textbf{P}_r,\textbf{X}_r,\textbf{z}_r,\textbf{s}_r,\textbf{w}_r)$ to represent the set of variables, where $\textbf{P}_r$, $\textbf{s}_r$ and $\textbf{w}_r$ are defined the same as in the aggregated formulation, $\textbf{X}_r$ represents the collection of variables $\{X^{\textrm{pump}}_{r,t,i}\}_{t\in\mathcal{T},i\in \{1,\ldots,k_r\}}$ and $\{X^{\textrm{gen}}_{r,t,i}\}_{t\in\mathcal{T},i\in \{1,\ldots,k^\prime_r\}}$, and $\textbf{z}_r$ the collection of variables $\{z^{\textrm{pump}}_{r,t,i}\}_{t\in\mathcal{T},i\in\{1,\ldots,k_r\}}$ and $\{z^{\textrm{gen}}_{r,t,i}\}_{t\in\mathcal{T},i\in \{1,\ldots,k^\prime_r\}}$. Using these notations, we define the following set:
\begin{equation}\label{eq31}
\mathcal{S}^{\prime\prime}_{r}=\left\{ (\textbf{P}_r,\textbf{X}_r,\textbf{z}_r,\textbf{s}_r,\textbf{w}_r) \,\left|\,
\begin{array}{@{}lll}
& \eqref{eq3},~\eqref{eq6},~\eqref{eq16},\\[3pt]
& \eqref{eq17},~\eqref{eq27},~\eqref{eq30}\\[3pt]
& \textbf{z}_r\in \{0,1\}^{T\times(k_r+k^\prime_r)} \\[3pt]
&\textbf{w}_r\in \{0,1\}^{2T}
\end{array}\right.\right\}.
\end{equation}
The reason that we call the formulation $\mathcal{S}^{\prime\prime}_{r}$ as presolved formulation is that its definition is based on a preprocessing step by which the sets $\mathcal{P}^{\textrm{pump}}_{r}$ and $\mathcal{P}^{\textrm{gen}}_{r}$ are reformulated as the unions of disjoint intervals. This preprocessing step is effective for breaking the symmetric structure, since we only keep the variables $\textbf{P}_r$ to represent the total amount of generating and the total amount of pumping load, and drop all the variables that represent the status of PSH units in reservoir $r$. In this way, the solutions that achieve the same amount of total pumping load and the same amount of total generation via different status of PSH units are composed into one solution in the presolved formulation, thus the symmetric structures is broken.

Using the presolved formulation, problem \eqref{eq11} can be reformulated as follows:
\begin{equation}\label{eq32}
\begin{aligned}
\min~&\sum_{g\in\mathcal{G}} C(\textbf{u}_{g},\textbf{q}_{g})\\
%\textrm{s.t.}~&\sum_{g\in\mathcal{G}} q_{g,t}+ \sum_{r\in\mathcal{R}}P^{\textrm{gen}}_{r,t}- \sum_{r\in\mathcal{R}}P^{\textrm{pump}}_{r,t}=D_t,\\
\textrm{s.t.}~&\sum_{g\in\mathcal{G}} q_{g,t}+ \sum_{r\in\mathcal{R}}(P^{\textrm{gen}}_{r,t}- P^{\textrm{pump}}_{r,t})=D_t,~t\in\mathcal{T},\\
&(\textbf{P}_r,\textbf{X}_r,\textbf{z}_r,\textbf{s}_r,\textbf{w}_r)\in\mathcal{S}^{\prime\prime}_r,~g\in\mathcal{G}_r,\\
&(\textbf{u}_{g},\textbf{q}_{g})\in \mathcal{X}_g,~g\in\mathcal{G}.
\end{aligned}
\end{equation}

To construct the presolved formulation, we should reformulate $\mathcal{P}^{\textrm{pump}}_{r}$ and $\mathcal{P}^{\textrm{gen}}_{r}$ as the union of disjoint intervals in \eqref{eq26} and \eqref{eq29}. For this purpose, we design an algorithm which can reformulate the union of $k$ intervals into the union of disjoint intervals. The algorithm is described in Algorithm 1.

\begin{algorithm}[t]
\caption{An Algorithm for Merging Intervals}
\small{
\begin{algorithmic}[1]
\STATE \textbf{Input:} A set of intervals $[l_1,u_1],\cdots,[l_k,u_k]\subseteq \mathbb{R}$.
\STATE Using a sorting algorithm to reorder the intervals to the sequence such that $l_1\leq l_2\leq\cdots\leq l_k$.
\STATE Initialize $i=1$, $[l^\prime_1,u^\prime_1]=[l_1,u_1]$, and $\mathcal{L}=\emptyset$.
\FOR{$j=2,\cdots,k$}
\IF{$l_j\leq u^\prime_i$}
\STATE Update $u^\prime_i \leftarrow \max\{u^\prime_i,u_j\}$. %//Merge $[l^\prime_i,u^\prime_i]$ and $[l_j,u_j]$
\ELSE
\STATE Update $\mathcal{L} \leftarrow \mathcal{L} \bigcup [l^\prime_i,u^\prime_i]$.
\STATE Set $[l^\prime_{i+1},u^\prime_{i+1}]=[l_j,u_j]$ .
\STATE Update $i=i+1$.
\ENDIF
\ENDFOR
\STATE Update $\mathcal{L} \leftarrow \mathcal{L} \bigcup [l^\prime_i,u^\prime_i]$ and set $k^\prime=i$.
\STATE \textbf{Output:} The nonintersecting intervals $[l^\prime_1,u^\prime_1],\cdots,[l^\prime_{k^\prime},u^\prime_{k^\prime}]$.
\end{algorithmic}}
\label{alg1}
\end{algorithm}

To show the intervals $[l^\prime_1,u^\prime_1],\cdots,[l^\prime_{k^\prime},u^\prime_{k^\prime}]$ generated by Algorithm \ref{alg1} do not intersect with each other, we prove the following proposition:
\begin{proposition}
The intervals $[l^\prime_1,u^\prime_1],\cdots,[l^\prime_{k^\prime},u^\prime_{k^\prime}]$ returned by Algorithm 1 satisfy $l^\prime_1\leq u^\prime_1<l^\prime_2\leq u^\prime_2<\cdots<l^\prime_{k^\prime}\leq u^\prime_{k^\prime}$.
\end{proposition}
\begin{proof}
We prove this proposition by induction. First, for $i=1$, it is trivial to see that $l^\prime_1\leq u^\prime_1$ holds. Now, we assume that $l^\prime_1\leq u^\prime_1<l^\prime_2\leq u^\prime_2<\cdots<l^\prime_d\leq u^\prime_d$ for some $d\in\{1,\ldots, k^\prime-1\}$, and consider the interval $[l^\prime_{d+1},u^\prime_{d+1}]$.
%Letting $j^\ast$ be the value of the variable $j$
Note that the interval $[l^\prime_{d+1},u^\prime_{d+1}]$ is generated at the moment when Algorithm \ref{alg1} is running to Line 9 with $i=d$. At this moment, the condition in Line 5 is not satisfied. In other words, we have $l_{j}> u^\prime_i=u^\prime_d$. Meanwhile, the interval $[l^\prime_{d+1},u^\prime_{d+1}]$ is set to $[l_j,u_j]$ at the same moment. Thus, we have $l^\prime_{d+1}> u^\prime_{d}$.
%in the loop where $[l^\prime_{i},u^\prime_{i}]$ is merged into $\mathcal{L}$. Clearly, we have $u^\prime_{i}<l_{j^\ast}=l^\prime_{i+1}\leq u^\prime_{i+1}$.
Thus, by induction, the proposition is correct.
\end{proof}

Given $k$ intervals as the input, the complexity of Algorithm \ref{alg1} is $\mathcal{O}(k \log k)$. When applying Algorithm \ref{alg1} for constructing either $\mathcal{P}^{\textrm{gen}}_{r}$ or $\mathcal{P}^{\textrm{pump}}_{r}$, the number of intervals input to the algorithm is $2^{n_r}-1$. Thus, using Algorithm 1 as the presolving algorithm for constructing the presolved formulation \eqref{eq32}, the complexity is $\mathcal{O}(n_r\cdot 2^{n_r})$. Although the complexity is on the order of exponential, the value $n_r$ is generally very small. Thus,  in practice, the computational time for running this presolving is almost ignorable in comparing with the time of solving \eqref{eq32}.

The number of binary variables in the presolved formulation depends on the distributions of the maximum and minimum generation/pumping power of each unit. In the worst-case, the number of disjoint intervals for representing the set $\mathcal{P}^{\textrm{pump}}_{r}$ and $\mathcal{P}^{\textrm{gen}}_{r}$ is $2^{n_r}-1$. In this case, the presolved formulation is not compact. On the other hand, if the distributions of the maximum and minimum generation/pumping power of different units are close to each other, then the number of disjoint intervals $k^\prime$ returned by Algorithm \ref{alg1} is generally much fewer than $2^{n_r}-1$. For example, consider the representation in \eqref{eq25}. If the intersection $\bigcap_{g\in \mathcal{G}_r}\left[\underline{p}^{\textrm{pump}}_g,\bar{p}^{\textrm{pump}}_g\right]$ is nonempty and has a common point $p^{\ast}$, then the number of disjoint intervals in \eqref{eq26} is no more than $n_r$, since for each $s\in\{1,\ldots,n_r\}$, we have
\begin{equation}\label{eq33}
 p^{\ast}s\in\left[\sum_{g\in \mathcal{G}_r} \underline{p}^{\textrm{pump}}_g u^{\textrm{pump}}_{g,t},\sum_{g\in \mathcal{G}_r} \bar{p}^{\textrm{pump}}_g u^{\textrm{pump}}_{g,t}\right]
\end{equation}
for all $\textbf{u}^{\textrm{pump}}_{r,t}\in \mathcal{U}^{\textrm{pump}}_{r,t}$ such that $\sum_{g\in \mathcal{G}_r} u^{\textrm{pump}}_{g,t}=s$, and these intervals will be merged. The above analysis explains why the presolved formulation is generally compact in practice.

\begin{table}[h]
\begin{center}
\caption{Parameters of the PSH unit.}\label{tab1}
\begin{tabular}{cccccc}
\toprule
$\underline{p}^{\textrm{gen}}_g$ &$\bar{p}^{\textrm{gen}}_g$ &$\underline{p}^{\textrm{pump}}_g$ &$\bar{p}^{\textrm{pump}}_g$  &\multirow{2}*{$\alpha_g$} &\multirow{2}*{$\beta_g$}\\
(MW) &(MW) &(MW) &(MW)  &~ &~\\
\midrule
100 &200 &195 &205 &0.9 &0.9\\
\bottomrule
\end{tabular}
\end{center}
\end{table}

\begin{table}[h]
\begin{center}
\caption{Reservoid parameters for our case study.}\label{tab2}
\begin{tabular}{cccc}
\toprule
$\underline{S}_r$ &$\bar{S}_r$ &$s_{r,0}$ &$S_{r,T}$ \\
(MWh) &(MWh) &(MWh) &(MWh) \\
\midrule
1000 &3500 &2600 &2600 \\
\bottomrule
\end{tabular}
\end{center}
\end{table}

%
%\begin{table}[h]
%\begin{center}
%\caption{Reservoir parameters for our case study.}\label{tab2}
%\begin{tabular}{cccc}
%\toprule
%$\underline{S}_r$ &$\bar{S}_r$ &$S_{r,0}$ &$S_{r,T}$ \\
%(MWh) &(MWh) &(MWh) &(MWh) \\
%\midrule
%1000 &3500 &2600 &2600 \\
%\bottomrule
%\end{tabular}
%\end{center}
%\end{table}

At the end of this section, we provide an example to illustrate the two new formulations. Consider the case in which the plant have $3$ identical units sharing the same reservoir. The parameters of the identical PSH units are given in Table \ref{tab1}, and some other parameters of the plant are given in Table \ref{tab2}. Then, the constraints in the aggregated formulation are given as follows:
\begin{equation}\label{eq34}
\begin{aligned}
&U^{\textrm{gen}}_{r,t}+U^{\textrm{pump}}_{r,t}+U^{\textrm{off}}_{r,t}=3,~t\in \mathcal{T},\\
&100 U^{\textrm{gen}}_{r,t} \leq P^{\textrm{gen}}_{r,t}\leq 200 U^{\textrm{gen}}_{r,t},~t\in \mathcal{T},\\
&195 U^{\textrm{pump}}_{r,t}\leq P^{\textrm{pump}}_{r,t}\leq 205 U^{\textrm{pump}}_{r,t},~t\in \mathcal{T},\\
&w_{r,t}^{\textrm{pump}}+w_{r,t}^{\textrm{gen}}\leq 1,~t\in \mathcal{T},\\
&U_{r,t}^{\textrm{pump}}\leq 4 w_{r,t}^{\textrm{pump}},~t\in \mathcal{T},\\
&U_{r,t}^{\textrm{gen}}\leq 4 w_{r,t}^{\textrm{gen}},~t\in \mathcal{T},\\
&\eqref{eq6},~\eqref{eq16},~\eqref{eq17},\\
&U^{\textrm{pump}}_{r,t},~U^{\textrm{gen}}_{r,t}\in\mathbb{Z},~t\in \mathcal{T},\\
&w_{r,t}^{\textrm{gen}},~w_{r,t}^{\textrm{pump}}\in\{0,1\},~t\in \mathcal{T},
\end{aligned}
\end{equation}
and the constraints in the presolved formulation are given as follows:
\begin{equation}\label{eq35}
\begin{aligned}
&P^{\textrm{gen}}_{r,t}=X^{\textrm{gen}}_{r,t,1},~t\in \mathcal{T},\\
&100 z^{\textrm{gen}}_{r,t,1}\leq X^{\textrm{gen}}_{r,t,1}\leq 600 z^{\textrm{gen}}_{r,t,1},~t\in \mathcal{T},\\
&P^{\textrm{pump}}_{r,t}= X^{\textrm{pump}}_{r,t,1}+ X^{\textrm{pump}}_{r,t,2}+ X^{\textrm{pump}}_{r,t,3},~t\in \mathcal{T},\\
&195 z^{\textrm{pump}}_{r,t,1}\leq X^{\textrm{pump}}_{r,t,1}\leq 205 z^{\textrm{gen}}_{r,t,1},~t\in \mathcal{T},\\
&390 z^{\textrm{pump}}_{r,t,2}\leq X^{\textrm{pump}}_{r,t,2}\leq 410 z^{\textrm{gen}}_{r,t,2},~t\in \mathcal{T},\\
&585 z^{\textrm{pump}}_{r,t,3}\leq X^{\textrm{pump}}_{r,t,3}\leq 615 z^{\textrm{gen}}_{r,t,3},~t\in \mathcal{T},\\
&w_{r,t}^{\textrm{pump}}+w_{r,t}^{\textrm{gen}}\leq 1,~t\in \mathcal{T},\\
&z^{\textrm{gen}}_{r,t,1}=w_{r,t}^{\textrm{gen}},~t\in \mathcal{T},\\
&z^{\textrm{pump}}_{r,t,1}+z^{\textrm{pump}}_{r,t,2}+z^{\textrm{pump}}_{r,t,3}=w_{r,t}^{\textrm{pump}},~t\in \mathcal{T},\\
&\eqref{eq6},~\eqref{eq16},~\eqref{eq17},\\
&z^{\textrm{gen}}_{r,t,1},~z^{\textrm{pump}}_{r,t,1},~z^{\textrm{pump}}_{r,t,2},~z^{\textrm{pump}}_{r,t,3}\in\{0,1\},~t\in \mathcal{T},\\
&w_{r,t}^{\textrm{gen}},~w_{r,t}^{\textrm{pump}}\in\{0,1\},~t\in \mathcal{T}.\\
\end{aligned}
\end{equation}

\section{Numerical Experiments}\label{sec4}

In this section, we carry out numerical studies on comparing the three formulations, including the standard formulation, the aggregated formulation, and the presolved formulation, to illustrate the value of different PSH formulations to the computational efficiency of Gurobi, a well-known state-of-the-art commercial solver. Our numerical experiments are carried out on a personal computer with 9th Gen Intel(R) Core(TM) i7-9700 3.0 GHz CPU and 16 GB RAM. All the mixed-integer programming problems in different formulations are solved using Gurobi 10.0.

\subsection{Unit Commitment with Identical PSH Units}\label{subsec41}

In this subsection, we consider the cases of unit commitment problem that have multiple identical PSH units being incorporated. We assume that these identical PSH units share the same reservoir. We evaluate the impact of the three formulations in \eqref{eq11}, \eqref{eq22} and \eqref{eq32}. The test instances of the unit commitment problems used in our numerical studies have been generated in \cite{Bacci} and are publicly available\footnote{The test instances are selected from the set DS1 in \cite{Bacci}, which are downloaded from \url{https://commalab.di.unipi.it/files/Data/UC/RCUC2.tgz}}. However, these instances only contain pure thermal units. Hence, in addition to the thermal units, some identical PSH units are incorporated into the test instances. The parameters of the PSH units, and the parameters of the reservoir used in our case study are presented in Table \ref{tab1}.

We solve the test instances using the three formulations, including the standard formulation, the aggregated formulation, and the presolved formulation. In our experiments, the parameters OptimalityTol and MIPGap of Gurobi are all set to $10^{-4}$, and the default symmetric breaking techniques implemented in Gurobi are turned-off (by setting the parameter ``Symmetry'' in Gurobi to 0). The time limit is set to 1800 seconds. Besides of the above configurations, we also run another implementation, named as Standard+Sym, using the standard formulation, with the default symmetric breaking techniques in Gurobi turned-on.

We consider 16 settings that have different number of thermal units $n$  (10 and 20), different time periods $T$ (24 and 36), and different number of identical PSH units (ranging from 1 to 4). In each setting, there are 5 instances being solved. The average of the results for each setting are obtained and listed in Table \ref{tab4}.

\begin{table*}[t]
\begin{center}
\caption{Numerical Results on Unit Commitment Problems with Identical PSH Units}\label{tab4}
\resizebox{\textwidth}{!}{%
\begin{tabular}{ccrrrrrrrrrrrr}
\toprule
\multirow{2}*{$T$} &\multirow{2}*{$n$} &\multicolumn{3}{c}{Standard} &\multicolumn{3}{c}{Aggregated} &\multicolumn{3}{c}{Presolved} &\multicolumn{3}{c}{Standard+Sym}\\
\cline{3-5}\cline{6-8}\cline{9-11}\cline{12-14}
~ & ~ &solved &time(s) &nodes &solved &time(s) &nodes &solved &time(s) &nodes &solved &time(s) &nodes\\
\midrule
\multicolumn{14}{c}{Number of Identical PSH Units = 1}\\
\midrule
24 &10 &5 &0.816 &182.8  &5 &0.83 &182.8 &5 &0.776 &170.4 &5 &0.820 &182.8 \\
24 &20 &5 &2.946 &1320.8 &5 &2.866 &1324.8 &5 &2.876 &1279.2 &5 &2.902 &1324.8\\
36 &10 &5 &1.375 &303.5 &5 &1.390 &303.5 &5 &1.373 &263.5 &5 &1.393 &303.5 \\
36 &20 &5 &14.892 &3698.0 &5 &14.732 &3643.6 &5 &11.802 &3453.6 &5 &14.548 &3643.6 \\
\midrule
\multicolumn{14}{c}{Number of Identical PSH Units = 2}\\
\midrule
24 &10  &5 &1.784 &4460.8 &5 &0.808 &381.8 &5 &1.010 &470.6 &5 &1.090 &532.8\\
24 &20  &5 &4.006 &2750.8 &5 &3.390 &1957.6 &5 &3.386 &1973.6 &5 &3.968 &2117.0 \\
36 &10  &5 &2.284 &2310.4 &5 &1.296 &227.2 &5 &1.670 &570.0 &5 &1.702 &397.4\\
36 &20  &5 &8.408 &3094.4 &5 &7.402 &1249.4 &5 &7.396 &1373.6 &5 &7.710 &1236.8\\
\midrule
\multicolumn{14}{c}{Number of Identical PSH Units = 3}\\
\midrule
24 &10 &5 &7.484 &42167.4 &5 &0.854 &338.4  &5 &0.884 &647.2 &5 &1.376 &574.4  \\
24 &20 &5 &13.284 &34563.0 &5 &3.462 &2499.0 &5 &3.846 &2499.4 &5 &6.050 &5945.2 \\
36 &10 &3 &64.523 &314746.0 &5 &1.404 &321.0 &5 &2.282 &1156.0 &5 &11.098 &3651.8 \\
36 &20 &4 &364.825 &713711.3 &5 &7.910 &1207.6 &5 &8.830 &2307.4 &5 &9.592 &1739.0 \\
\midrule
\multicolumn{14}{c}{Number of Identical PSH Units = 4}\\
\midrule
24 &10  &5 &45.808 &606641.8 &5 &0.790 &392.0 &5 &1.240 &814.8 &5 &1.356 &561.6  \\
24 &20  &4 &41.760 &107264.3 &5 &3.652 &2640.6 &5 &4.866 &3504.8 &5 &6.648 &4223.2\\
36 &10  &1 &1165.650 &5191956.0 &5 &1.528 &363.6 &5 &3.998 &1972.4 &5 &3.652 &3287.4\\
36 &20  &2 &924.820 &2170526.5 &5 &7.250 &978.4 &5 &9.500 &2670.2 &5 &13.296 &2242.6\\
\bottomrule
\end{tabular}
}
\end{center}
\end{table*}

\begin{figure}[h]
\centering
\includegraphics[width=9.5cm]{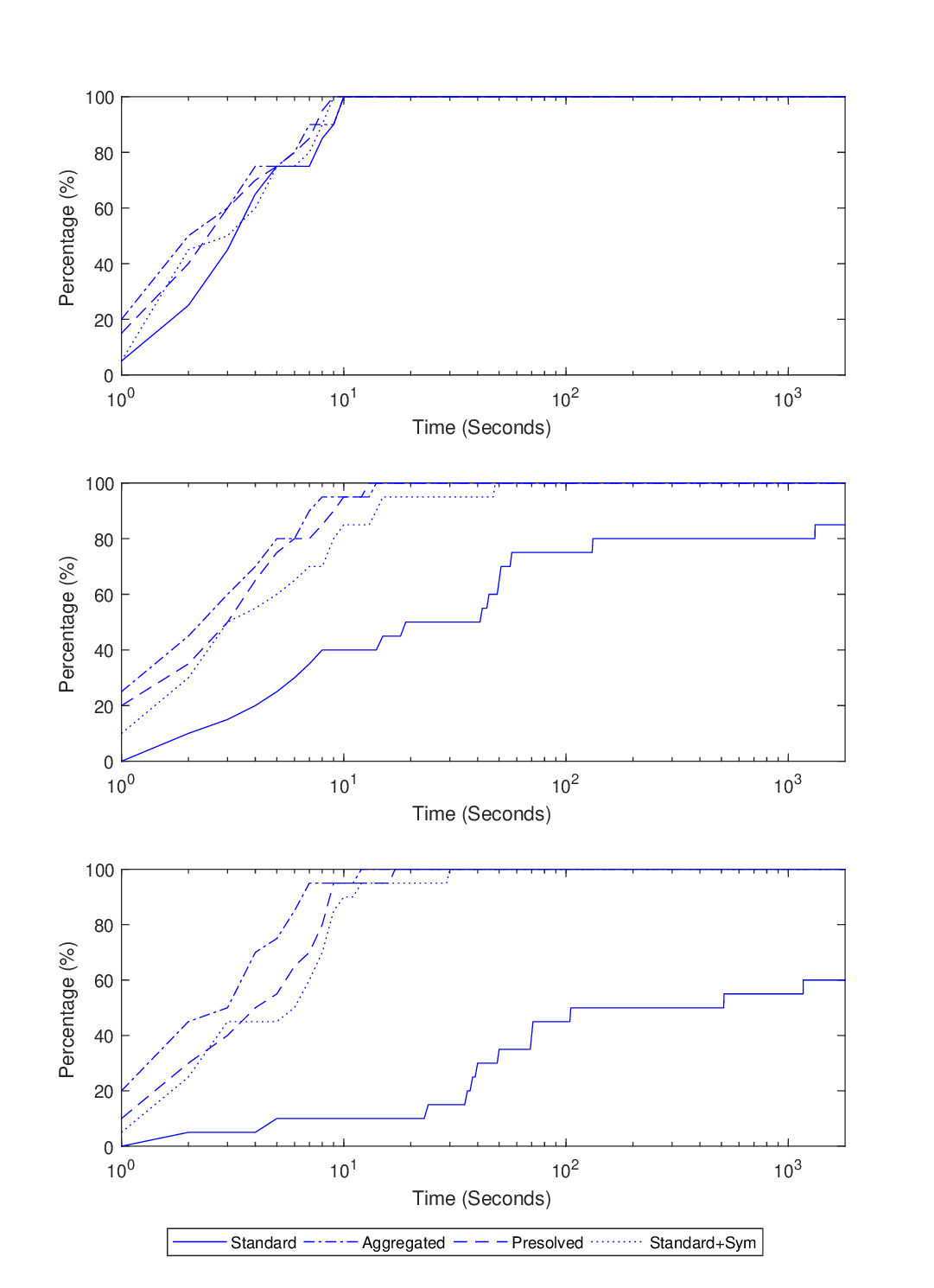}
\caption{The performance profile of using different formulations. From top to bottom, the three figures illustrate the case in which there are 2, 3 and 4 identical PSH units being incorporated, respectively.}\label{fig1}
\end{figure}

From the results listed in Table \ref{tab4}, we can see that the four implementations have similar performance when there is only one PSH unit in the reservoir. This is reasonable because the three formulations are almost the same in this case, and there is no symmetric structure in the problem. On the other hand, as the number of identical PSH units increases, the efficiency of the four implementations become quite different. For the standard formulation without symmetric breaking, the efficiency of the solver is much lower than the other three implementations. As we can see from the results under the column of ``Standard'' in Table \ref{tab4}, not only the number of problems solved within 1800 seconds is fewer, but also the average computational time is much longer, especially when the number of identical PSH units is larger than or equal to three. This observation indicates that the symmetric structures in the standard formulation have a significant impact on the efficiency of a solver.

In comparison, when using the two new formulations, the average computational time can be much shorter than the implementation of using the standard formulation. Moreover, for most cases, the implementations of using the two new formulations perform even better than the implementation of using the standard formulation together with the default symmetric breaking techniques in Gurobi. The above comparison results show that the two newly proposed formulations are very effective for breaking the symmetric structure, which can improve the efficiency of a solver effectively.

To further illustrate the performance of the four implementations, we also plot the performance profile in Figure \ref{fig1}. We consider the three cases where the numbers of identical PSH units ranges from $2$ to $4$. For each case, there are 20 test instances (corresponding to the 4 settings that have different number of thermal units and different number of periods). Figure \ref{fig1} illustrates the percentage of the test instances that can be solved within various time limit. From the results illustrated in Figure \ref{fig1}, we can see that the aggregated formulation based implementation performs best, and the presolved formulation based implementation also performs better than the standard formulation with the default symmetric breaking techniques. These results indicate that using the two new formulations to break the symmetric structure in the problem is more powerful than using the default symmetric breaking techniques in Gurobi to handle the standard formulation. Moreover, we would like to mention that the two new formulations can be easily applicable to other open-source solvers that do not support a powerful symmetric breaking technique as in Gurobi.

\subsection{Unit Commitment with Nonidentical PSH units}

In this subsection, we consider the case of unit commitment problem in which the PSH units are not identical, but have the same generating efficiency and pumping efficiency. In this case, since the aggregated formulation is not applicable, we only compare the standard formulation with the presolved formulation.

The test instances used in this section is similar to those in subsection \ref{subsec41}, except that the PSH units being incorporated are not identical. The parameters of the PSH units used in our numerical studies are listed in Table \ref{tab3}, in which only the parameters of generating and pumping efficiency are identical. We solve the two formulations using Gurobi with both parameters OptimalityTol and MIPGap of Gurobi being set to $10^{-4}$, and the time limit is set to 1800 seconds. The default symmetric breaking techniques are turned-on (In fact, the default symmetric breaking techniques have only marginal impacts on both formulations).

\begin{table}[t]
\begin{center}
\caption{Parameters of the PSH units.}\label{tab3}
\begin{tabular}{ccccccc}
\toprule
\multirow{2}*{Unit ID} &$\underline{p}^{\textrm{gen}}_g$ &$\bar{p}^{\textrm{gen}}_g$ &$\underline{p}^{\textrm{pump}}_g$ &$\bar{p}^{\textrm{pump}}_g$  &\multirow{2}*{$\alpha_g$} &\multirow{2}*{$\beta_g$}\\
~  &(MW) &(MW) &(MW) &(MW)  &~ &~\\
\midrule
1 &95 &195 &195 &205 &0.9 &0.9\\
2 &100 &200 &193 &208 &0.9 &0.9\\
3 &105 &205 &192 &204 &0.9 &0.9\\
4 &110 &210 &197 &210 &0.9 &0.9\\
\bottomrule
\end{tabular}
\end{center}
\end{table}

We consider 12 settings that have different number of thermal units $n$, different time periods $T$, and different number of PSH units. For an instance that have $k$ PSH units (where $k$ ranges from 2 to 4), the parameters of the first $k$ units listed in Table \ref{tab3} are used to construct the instance. In each setting, there are 5 instances. The average of the results for each setting are listed in Table \ref{tab5}. The performance profile for comparing the two formulations is shown in Figure \ref{fig2}.

\begin{table}[h]
\begin{center}
\caption{Numerical Results on Unit Commitment Problems with Nonidentical PSH Units}\label{tab5}
\begin{tabular}{ccrrrrrrrrrrrr}
\toprule
\multirow{2}*{$T$} &\multirow{2}*{$n$} &\multicolumn{3}{c}{Standard} &\multicolumn{3}{c}{Presolved}\\
\cline{3-5}\cline{6-8}
~ & ~ &solved &time(s) &nodes &solved &time(s) &nodes \\
\midrule
\multicolumn{8}{c}{Number of PSH Units = 2}\\
\midrule
24 &10 &5 &2.6 &8141.0 &5 &0.9 &577.4  \\
24 &20 &5 &3.5 &2415.0 &5 &3.1 &1764.4 \\
36 &10 &5 &2.8 &4398.6 &5 &1.7 &632.6\\
36 &20 &5 &8.4 &2883.4 &5 &7.6 &1464.2 \\
\midrule
\multicolumn{8}{c}{Number of PSH Units = 3}\\
\midrule
24 &10 &5 &4.7 &20348.0 &5 &1.2 &846.2 \\
24 &20 &5 &24.5 &70966.2 &5 &3.8 &3141.6\\
36 &10 &4 &23.2 &89689.3 &5 &3.1 &2002.0 \\
36 &20 &5 &84.1 &135605.6 &5 &11.2 &2401.2 \\
\midrule
\multicolumn{8}{c}{Number of PSH Units = 4}\\
\midrule
24 &10 &5 &15.3 &119290.8 &5 &1.1 &899.2 \\
24 &20 &4 &206.7 &387530.3 &5 &3.8 &2620.2 \\
36 &10 &2 &74.0 &465312.5 &5 &3.3 &1079.4 \\
36 &20 &2  &28.1 &30725.0 &5 &19.3 &4257.4\\
\bottomrule
\end{tabular}
\end{center}
\end{table}

From the results in Table \ref{tab5}, we can see that for the standard formulation, even when there are as few as three or four PSH units in the unit commitment problem, the efficiency of solving a test instance using Gurobi is quite low. As we can see in the column ``solved'' of the standard formulation, there are several test instances that can not be solved within 1800 seconds. The above observations indicate that incorporating PSH units into a unit commitment problem will bring a great challenge on the efficiency of a solver if the problem formulation is not well designed. Moreover, the default symmetric breaking techniques in Gurobi do not work in these cases.

On the other hand, the results in Table \ref{tab5} show that when using the presolved formulation, all the test instances can be solved within a reasonable computational time, and the average computational time can be much shorter than the time of using the standard formulation. These comparison results indicate that using the presolved formulation is effective on handling the partial symmetric structures embedded in the problem. %, so that the efficiency of a solver can be improved significantly.

\begin{figure}[!t]
\centering
\includegraphics[width=9.5cm]{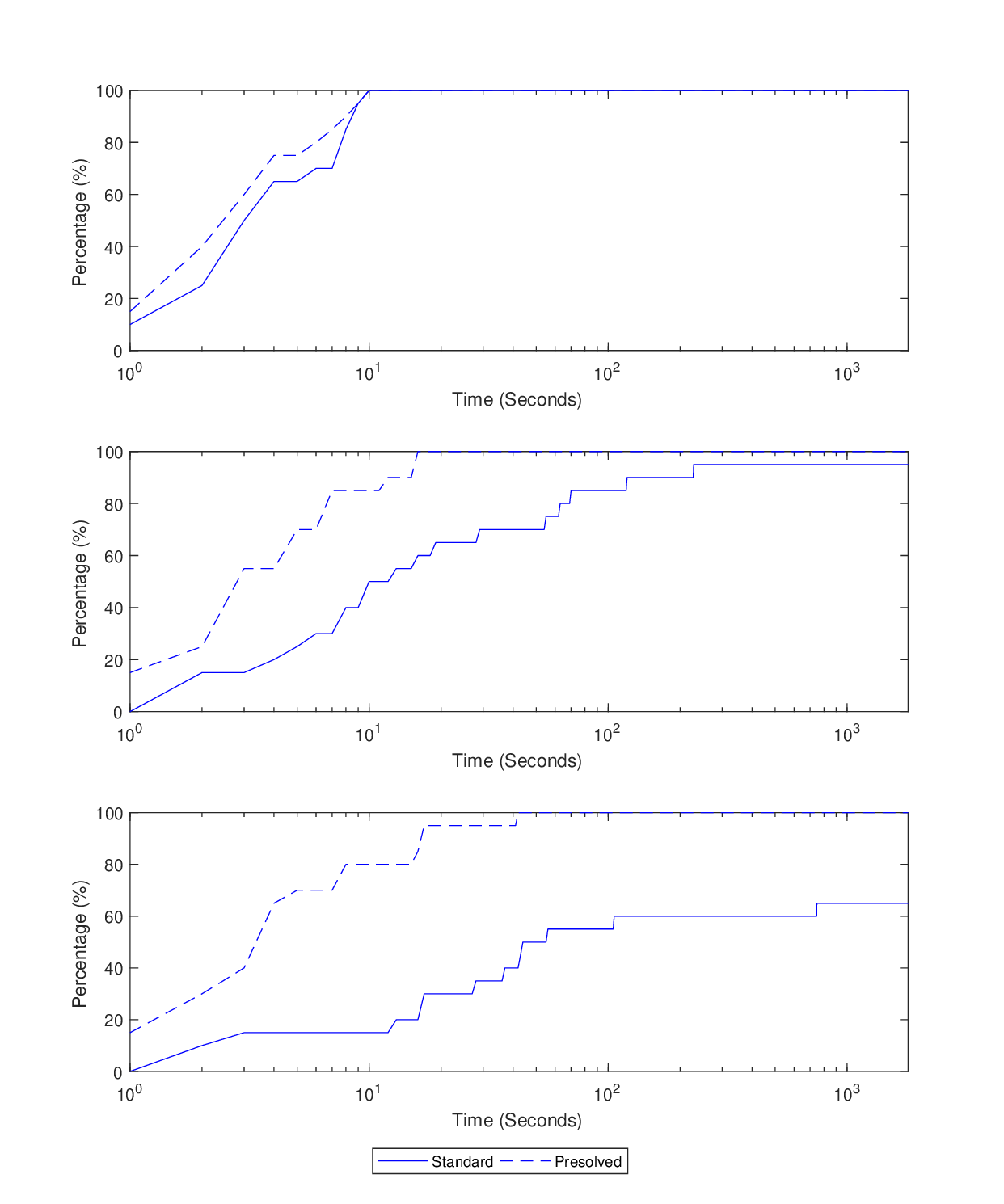}
\caption{The performance profile of using different formulations. From top to bottom, the three figures illustrate the case in which there are 2, 3 and 4 PSH units being incorporated, respectively.}\label{fig2}
\end{figure}

\section{Conclusion}
In this paper, we consider two types of symmetric structures arisen from the PSH units in a unit commitment problem, including the full symmetric structure, and the partial symmetric structure. We have shown that even when there are only three or four PSH units in a unit problem, the efficiency of a solver can degrade significantly if the symmetric structures in the problem are not taken into consideration.

To break the symmetric structures in the problem, we have proposed two new formulations: One is called as aggregated formulation, which can be applied to the case of full symmetric structure, and the other is called as presolved formulation, which can be applied to both the cases of full and partial symmetric structure. Our numerical experiments on some publicly available test instances indicate that the proposed formulations can break the symmetric structure in the problem effectively, and improve the efficiency of a solver significantly.

{\appendix[Formulations of the unit commitment problem]

In the unit commitment models \eqref{eq11}, \eqref{eq22} and \eqref{eq32}, we have applied the 3-bin formulation in \cite{Rajan} to formulate the constraint $(\textbf{u}_{g},\textbf{q}_{g})\in \mathcal{X}_g$. We provide its details in this appendix for the completeness of this paper.

The lower and upper bounds constraints on the powers are imposed as follows:
\begin{equation}\label{eq36}
\underline{q}_g x_{g,t} \leq q_{g,t}\leq \bar{q}_g x_{g,t},~\forall g\in \mathcal{G},t\in\mathcal{T}.\\
\end{equation}
The state transition of a unit can be formulated as follows:
\begin{equation}\label{eq37}
x_{g,t}-x_{g,t-1}=v_{g,t}-w_{g,t},~\forall g\in \mathcal{G},t\in\mathcal{T},\\
\end{equation}
where $x_{g,0}$ is a given constant that represents the initial condition of unit $g$.
The minimum up-time and down-time constraints can be formulated as follows:
\begin{equation}\label{eq38}
\begin{aligned}
&\sum_{s= t-\tau^{+}_{g}+1}^t v_{g,s}\leq x_{g,t},& g\in \mathcal{G},~t\in\{\tau_g^{+},\ldots,T\},\\
&\sum_{s= t-\tau^{-}_{g}+1}^t w_{g,s}\leq 1-x_{g,t},& g\in \mathcal{G},~t\in\{\tau_g^{-},\ldots,T\}.
\end{aligned}
\end{equation}
The constraints to specify initial conditions of a unit $g$ are imposed. We use $\tau_g^0$ to denote the initial state of the unit $g$, which is defined as follows: If $\tau_g^0>0$, then the unit has been in the on-state for $\tau_g^0$ time periods. On the other hand, if $\tau_g^0\leq 0$, then the unit has been in the off-state for $-\tau_g^0$ time periods.
For the case of $1<\tau_i^0<\tau_g^+$, the minimum up-time constraints can be defined by
\begin{equation}\label{eq39}
x_{g,t}=1,~t\in \{1,\ldots,\tau_g^{+} -\tau_g^0\}.
\end{equation}
Similarly, for the case of $1<-\tau_i^0<\tau_i^-$, the minimum down-time constraints can be defined by
\begin{equation}\label{eq40}
x_{g,t}=0,~t\in\{1,\ldots,\tau_g^{-} +\tau_g^0\}.
\end{equation}
The ramping constraints, together with the special treatments required by the start-up and shut-down periods, are formulated as follows:
\begin{equation}\label{eq41}
\begin{aligned}
&q_{g,t}-q_{g,t-1}\leq \Delta_g^{+} x_{g,t-1} + \Lambda_g^{+} v_{g,t}, &g\in \mathcal{G},~t\in\mathcal{T},\\
&q_{g,t-1}-q_{g,t}\leq \Delta_g^{-} x_{g,t} + \Lambda_g^{-} w_{g,t},&g\in \mathcal{G},~t\in\mathcal{T}.
\end{aligned}
\end{equation}

Given the above constraints, the set $\mathcal{X}_{g}$ is given as follows:
\begin{equation}\label{eq42}
\mathcal{X}_{g}=\left\{ (\textbf{u}_g,\textbf{q}_g) \,\left|\,
\begin{array}{@{}lll}
& \eqref{eq36}-\eqref{eq41},\\[3pt]
& \textbf{u}_g\in \{0,1\}^{3T} \\[3pt]
&\textbf{q}_g\in \mathbb{R}^{T}
\end{array}\right.\right\},
\end{equation}
where
\begin{equation}\label{eq43}
\begin{aligned}
&\textbf{u}_g=\left\{x_{g,t},v_{g,t},w_{g,t}\right\}_{g\in \mathcal{G},t\in \mathcal{T}},\\
&\textbf{q}_g=\left\{q_{g,t}\right\}_{g\in \mathcal{G},t\in \mathcal{T}}.
\end{aligned}
\end{equation}
Finally, the cost function of a unit $g$ is defined as follows:
\begin{equation}\label{eq44}
C(\textbf{u}_{g},\textbf{q}_{g})=\sum_{t\in\mathcal{T}} a_g q^2_{g,t} + b_g q_{g,t} +c_g x_{g,t} +s_g v_{g,t}.
\end{equation}}

\bibliographystyle{IEEEtran}

\bibliography{PSH}
\vfill
\end{document}